\documentclass{amsart}

\newtheorem{theorem}{Theorem}[section]
\newtheorem{lemma}[theorem]{Lemma}
\newtheorem{proposition}[theorem]{Proposition}

\newtheorem{corollary}[theorem]{Corollary}

\theoremstyle{definition}

\usepackage{amsmath,amssymb,amsfonts}

\theoremstyle{remark}

\numberwithin{equation}{section}

\renewcommand{\P}{\mathbb{P}}
\newcommand{\R}{\mathbb{R}}
\newcommand{\N}{\mathbb{N}}

\newcommand{\X}{\mathrm{X}}

\newcommand{\Y}{\mathrm{Y}}

\newcommand{\B}{\mathbf{B}}

\renewcommand{\S}{\mathbf{S}}

\renewcommand{\mod}{/}

\begin{document}

\title{The isometry group of $L^{p}(\mu,\X)$ is SOT-contractible}
\author{Jarno Talponen}
\address{University of Helsinki, Department of Mathematics and Statistics, Box 68, (Gustaf H\"{a}llstr\"{o}minkatu 2b) FI-00014 University 
of Helsinki, Finland}
\email{talponen@cc.helsinki.fi}

\date{\today}

\begin{abstract}
We will show that if $(\Omega,\Sigma,\mu)$ is an atomless positive measure space, $\X$ is a Banach space and $1\leq p<\infty$, then
the group of isometric automorphisms on the Bochner space $L^{p}(\mu,\X)$ is contractible in the strong operator topology.
We do not require $\Sigma$ or $\X$ above to be separable.
\end{abstract}
\maketitle

\section{Introduction}
This article deals with the topological structure of isometry groups of Banach spaces. Recall that an isometry group
$\mathcal{G}_{\X}$ of a Banach space $\X$ consists of linear isometric isomorphisms $T\colon \X\rightarrow\X$.   

The connectedness of groups of linear automorphisms with respect to the norm topology is a classical topic by now, see
e.g. \cite{Ku,Mit}. For example, Kuiper proved already in 1965 that the linear automorphism group $\mathrm{GL}_{\ell^{2}}$ 
and the isometry group $\mathcal{G}_{\ell^{2}}$ of $\ell^{2}$ are operator norm contractible. 
On the other hand, spaces $\mathcal{G}_{\ell^{p}}$ and $\mathcal{G}_{L^{p}(\mu)}$ are discrete in the operator 
norm topology for $1\leq p\leq \infty,\ p\neq 2,$ (\cite[p.112]{LT}, \cite[p.57]{Ca}) and a fortiori not contractible. 
Our main result result shows that the situation is surprisingly different when observing $\mathcal{G}_{L^{p}(\mu)}$ 
in the strong operator topology:
\begin{theorem}\label{thm1}
Let $\X$ be a Banach space, $\mu$ an atomless positive measure and $1\leq p<\infty$. 
Then the isometry group $\mathcal{G}_{L^{p}(\mu,\X)}$ of the Bochner space $L^{p}(\mu,\X)$ 
endowed with the strong operator topology is contractible. Here $\X$ and $L^{p}(\mu,\X)$ can be regarded as 
real or complex spaces.
\end{theorem}
This result involves the isometric structure and more precisely the $L^{p}$-structure of Bochner spaces.
For a fascinating treatment of the latter topic see \cite{LP,Gr} which are implicitly applied in this work.
Our study is also closely related to \cite{Gr_infty}, \cite{GJ} and \cite{GJK}. 

Recall that for Banach spaces $\X$ and $\Y$ a projection $P\colon \X\oplus\Y\rightarrow \Y$ 
is called an $L^{p}$-projection for a given $p\in [1,\infty)$ if 
\[||(x,y)||_{\X\oplus\Y}^{p}=||x||_{\X}^{p}+||y||_{\Y}^{p}\quad \mathrm{for\ each}\ x\in \X,\ y\in \Y.\] 
Such projections commute and in fact the $L^{p}$-structure $\P_{p}(\X)$ of $\X$, i.e. the set of all $L^{p}$-projections on $\X$ 
can be regarded as a complete Boolean algebra (see \cite{LP}). 

In order to put our result in the right context, let us mention that our route to Theorem \ref{thm1} 
is via analyzing the $L^{p}$-structure of $L^{p}(\mu,\X)$. This structure depends on $\mu$ and the 
$L^{p}$-structure of $\X$, as one might expect. If $\X$ is separable, then one can write 
$\P_{p}(L^{p}(\mu,\X))=\Sigma\mod \mu\otimes \P_{p}(\X)$ in a suitable sense, see \cite{GJK}.
It might first seem, bearing the $L^{p}$-structures in mind, that an effective means
to analyze the connectedness properties of $\mathcal{G}_{L^{p}(\mu,\X)}$ would be applying 
suitably normalized $L^{p}$-decompositions of Bochner spaces. For example the $p$-integral module representation (see \cite{LP}) 
appears to be a natural tool for such an approach.

In fact, the $L^{p}$-structure of $L^{p}(\mu,\X)$ is simple to represent if $\X$ is a separable Banach space having only trivial 
$L^{p}$-structure and $1\leq p<\infty,\ p\neq 2$. 
Then the isometries $T\in \mathcal{G}_{L^{p}(\mu,\X)}$ can be represented, in a suitable sense, as
\begin{equation}\label{eq: *}
Tf(t)=\sigma_{t}\left(\frac{d(\mu\circ \phi^{-1})}{d\mu}(t)\right)^{\frac{1}{p}}f\phi^{-1}(t)\ \mathrm{for}\ f\in L^{p}(\Omega,\Sigma,\mu,\X),
\end{equation}
where $\sigma\colon \Omega\rightarrow \mathcal{G}_{\X};\ t\mapsto \sigma_{t}$ is strongly measurable and 
$\phi\colon \Sigma\mod\mu\leftrightarrow \Sigma\mod \mu$ is a Boolean
isomorphism (see \cite{Gr} and also \cite{Gr_infty}).
Moreover, there is in general a close connection between the $L^{p}$-structure of a given Lebesgue-Bochner space and the
corresponding representations of type \eqref{eq: *}.  
 
However, there exists an obstruction, that has to be dealt with. Namely, in the setting of Theorem \ref{thm1} 
the space $\X$ may have a rich $L^{p}$-structure and be non-separable,
so that typically the $L^{p}$-structure of $L^{p}(\mu,\X)$ is very complicated and \eqref{eq: *} fails. 
Another substantial difficulty is that if $\X$ has a trivial $L^{p}$-structure but is \emph{non-separable},
then the $L^{p}$-structure of $L^{p}(\mu,\X)$ is not known explicitly, nor is it known whether representation \eqref{eq: *} holds. 
Thus we note that even though the group $\mathcal{G}_{L^{p}(\mu,\X)}$ has not been classified in the sense of $L^{p}$-structures, 
we are unexpectedly able to extract enough information of $\P_{p}(L^{p}(\mu,\X))$ in order to establish a very strong connectedness 
condition in Theorem \ref{thm1}.  
  
The way around the described problems is to employ a suitable isometric representation for $L^{p}(\mu,\X)$, which rises from 
measure-theoretic observations.

To introduce the notations applied in this paper, $\X,\Y$ and $\Y$ stand for \emph{real} Banach spaces. 
The closed unit ball and the unit sphere of $\X$ are denoted by $\B_{\X}$ and $\S_{\X}$, respectively. 
We refer to \cite{Ha} and \cite{Lac} for necessary background information in measure theory and isometric 
theory of classical Banach spaces. It is also useful to get acquainted with the machinery appearing in \cite{Gr} regarding Bochner spaces.

If $\mathcal{F}\subset \mathcal{P}(\Omega)$ and $\Sigma_{0}\subset \Sigma$ is a $\sigma$-subring such that $\mathcal{F}\subset\Sigma_{0}$ 
and for all $A\in \Sigma_{0}$ there is a set $\{B_{n}|n\in\N\}$, where each $B_{n}$ is a countable intersection of suitable elements of 
$\mathcal{F}$, and $\bigcup_{n}B_{n}=A$, then we say that $\mathcal{F}$ \emph{$\sigma$-generates} $\Sigma_{0}$.  
Recall that the strong operator topology (SOT) on $L(\X)$ is the topology inherited from $\X^{\X}$ endowed with the product topology.
Recall that each $T\in L(\X)$ has a $\mathrm{SOT}$-neighbourhood basis consisting 
of sets of the following type:
\begin{equation}\label{eq: SOTbasis}
\{S\in L(\X): ||(S-T)x_{1}||,||(S-T)x_{2}||,\ldots,||(S-T)x_{n}||<\epsilon\}, 
\end{equation}
where $x_{1},\ldots,x_{n}\in \X,\ n\in \N$ and $\epsilon>0$.

Let us make some preparations towards the proof of Theorem \ref{thm1} by introducing some auxiliary notions. 
In what follows $(\Omega,\Sigma,\mu)$ is a positive measure space, where $\Sigma$ is a $\sigma$-ring.
An arbitrary measure space may be unconveniently rich for our purposes, that is, for studying the bands of $L^{p}(\Omega,\Sigma,\mu)$ for 
$p\in [1,\infty)$. Therefore we wish to extract exactly the information which is 'recoqnized' by the $L^{p}$-structure. 
If $\Sigma$ is as above, then we will define
\[\Sigma^{L}=\{f^{-1}(\R\setminus\{0\})|f\in L^{p}(\Omega,\Sigma,\mu)\}\mod \mu .\]
Above $\{f^{-1}(\R\setminus\{0\})|f\in L^{p}(\Omega,\Sigma,\mu)\}$ is a $\sigma$-subring of $\Sigma$ and
$\Sigma^{L}$ is the quotient $\sigma$-ring formed by identifying $\mu$-null sets with $\emptyset$. 
By slight abuse of notation we denote the corresponding measure ring by $(\Sigma^{L},\mu)$ and we consider the elements 
$A\in \Sigma^{L}$ as contained in $\Omega$ in the $\mu$-a.e. sense. Hence, for $A,B\in \Sigma^{L}$, we will write 
$A\bigcup B\in \Sigma^{L}$ instead of $A\bigvee B \in \Sigma^{L}$, and so on. 
Note that $\Sigma^{L}$ does not depend on the value of $p$, as long as $p<\infty$, 
since $\{f^{-1}(\R\setminus\{0\})|f\in L^{p}(\Omega,\Sigma,\mu)\}\subset \Sigma$ is just the $\sigma$-subring of
$\sigma$-finite sets. The motivation of this concept becomes clear in the subsequent results. 
Given $p\in [1,\infty)$, a set $I$ and Banach spaces $\X_{i}$ for $i\in I$, we denote by $\bigoplus_{i\in I}^{p}\X_{i}$
the $L^{p}$-sum of the spaces $\X_{i}$. That is, $(x_{i})_{i\in I}\in \prod_{i\in I} \X_{i}$
satisfies $(x_{i})_{i\in I}\in \bigoplus^{p}_{i\in I}\X_{i}$ if and only if 
$||(x_{i})_{i\in I}||^{p}\stackrel{\cdot}{=}\sum_{i\in I}||x_{i}||_{\X_{i}}^{p}<\infty$. The space $\bigoplus_{i\in I}^{p}\X_{i}$ 
is endowed with the complete norm $||\cdot||$. We will denote by 
$P_{j}\colon \bigoplus_{i\in I}^{p} \X_{i}\rightarrow \X_{j}$
the $L^{p}$-projection onto $\X_{j}$ for each $j\in I$, where $X_{i}$ is regarded
as a subspace of $\bigoplus_{i\in I}^{p}\X_{i}$ in the natural way. 
Hence each $x\in \X_{j}$ is thought of as an element of $\bigoplus_{i\in I}^{p}\X_{i}$ as $P_{j}x=x$. 

The Lebesgue measure on $[0,1]$ is denoted by $m$ and if $\kappa$ is a non-zero cardinal, then we denote the product measure on 
$[0,1]^{\kappa}$ by $m^{\kappa}\colon \Sigma_{[0,1]^{\kappa}}\rightarrow \R$, where $\Sigma_{[0,1]^{\kappa}}$ is the 
corresponding product $\sigma$-algebra.

The \emph{orbit} of $x\in \S_{\X}$ is $\mathcal{G}_{\X}(x)\stackrel{\cdot}{=}\{T(x)|T\in \mathcal{G}_{\X}\}$.

\section{Results}
We will require the following auxiliary results.

\begin{lemma}\label{fact2}
For given $f\in L^{p}(m,\X)$ and $t_{0}\in [0,1)$ it holds that 
\[\lim_{\substack{t\rightarrow t_{0}\\ t\in [0,1]}}\int_{t_{0}}^{1}||f(h_{t}(s))-f(h_{t_{0}}(s))||^{p}\ \mathrm{d}s=0,\]
where $h_{t}\colon [t_{0},1]\rightarrow [t,1];\ h_{t}(s)=\left(\frac{t-t_{0}}{1-t_{0}}+\frac{1-t}{1-t_{0}}s\right)$ for 
$t,s \in [0,1]$.
\end{lemma}
\begin{proof}
The claim reduces to the analogous scalar-valued statement by approximating $f$ with simple functions. 
This in turn can be obtained by a straightforward modification of the proof of the classical fact that
$\lim_{h\rightarrow 0}\int_{\R}|g(s+h)-g(s)|\ \mathrm{d}s=0$ for $g\in L^{p}(\R),\ p<\infty$,
which exploits Lusin's theorem.
\end{proof}

\begin{lemma}\label{thm2}
Let $\X$ be a Banach space, $(\Omega,\Sigma,\mu)$ be an atomless positive measure space and $p\in [1,\infty)$.
Then there exists a set $I$ such that $L^{p}(\mu,\X)$ is isometrically isomorphic to 
$\bigoplus_{i\in I}^{p}L^{p}(m,L^{p}(m^{\lambda_{i}},\X))$, 
where $\lambda_{i}$ are non-zero cardinals for $i\in I$.
\end{lemma}  
\begin{proof}
The argument is closely related to classical matters discussed in \cite{Lac} and \cite{GJK}. 
Recall Lamperti's classical result \cite{Lamperti} that the supports of $f,g\in L^{1}(\mu)$ are essentially disjoint 
if and only if $||f+g||+||f-g||=2(||f||+||g||)$. The crucial conclusion of this result is that the disjointness of two vectors 
$f,g\in L^{p}(\nu)$ can be detected by studying the above norms and hence each linear isometry 
$\psi\colon L^{1}(\nu)\rightarrow L^{1}(\upsilon)$, where $\nu$ and $\upsilon$ are positive measures, preserves disjointness and bands.
This also leads to the fact that a projection $P$ on $L^{1}(\mu)$ is $L^{1}$-projection with a separable range 
if and only if there is $A\in \Sigma^{L}$ such that the image of $P$ is 
$\{f\in L^{1}(\mu)|\mathrm{supp}(f)\subset A\ \mu-\mathrm{a.e.}\}$.

Since each $f\in L^{1}(\mu)$ is $\sigma$-finitely supported, it follows that the measure ring
$(\Sigma^{L},\mu)$ is $\sigma$-generated by $\mu$-finite sets. Recall that the Boolean algebra of $L^{1}$-projections
on $L^{1}(\mu)$ is complete (see \cite[Prop.1.6]{LP}). Thus, by recalling Lamperti's result we obtain that 
$\{A\setminus \bigcup \mathcal{F}|A\in \Sigma^{L}\}$ defines an ideal of $\Sigma^{L}$ for any $\mathcal{F}\subset \Sigma^{L}$. 
Hence Hausdorff's maximum principle yields a maximal family $\{V_{j}\}_{j\in J}\subset \Sigma^{L}$ of pairwise $\mu$-essentially 
disjoint $\mu$-finite sets. Note that $\bigoplus_{j\in J}^{1}L^{1}(\mu|_{V_{j}})$ is isometric to $L^{1}(\mu)$. 

By using the $\mu$-finiteness of the sets $V_{j}$ we obtain that for each $j\in J$ there is a countable set $A_{j}$ and cardinals 
$\lambda_{k}$ for $k\in A_{j}$ such that the subspace $\{f\in L^{1}(\mu)|\mathrm{supp}(f)\subset V_{j}\ \mu-\mathrm{a.e.}\}$ is isometric 
to $\bigoplus_{k\in A_{j}}^{1}L^{1}(m^{\lambda_{k}})$ (see \cite[p.127]{Lac}). 
Since $V_{j}$ are essentially disjoint, the sets $A_{j}$ can be chosen to be pairwise disjoint.

Put $I=\bigcup_{j\in J}A_{j}$. Observe that there exists an isometric isomorphims\\ 
$\phi\colon \bigoplus^{1}_{i\in I}L^{1}(m^{\lambda_{i}}) \rightarrow L^{1}(\mu)$.
According to Lamperti's result the map $\phi$ preserves bands.
Recall that if $(\Omega_{1},\Sigma_{1},\mu_{1})$ and $(\Omega_{2},\Sigma_{2},\mu_{2})$ 
are positive $\sigma$-finite measure spaces such that $L^{1}(\mu_{1})$ and $L^{1}(\mu_{2})$ are isometric, 
then there is a Boolean isomorphism 
$\tau\colon \Sigma_{1}\mod \mu_{1}\leftrightarrow \Sigma_{2}\mod \mu_{2}$ see e.g. \cite[p.477]{GJK}.
It follows from the selection of $\{V_{j}\}_{j\in J}$  
that there exist ideals $B_{i}\subset \Sigma^{L}$ for $i\in I$ such that 
\begin{enumerate}
\item[(a)]{$\Sigma^{L}$ is $\sigma$-generated by $\bigcup_{i\in I}B_{i}$,}
\item[(b)]{$B_{i}$ is Boolean isomorphic to $\Sigma_{[0,1]^{\lambda_{i}}}\mod m^{\lambda_{i}}$ for $i\in I$}
\item[(c)]{the ideals $B_{i}$ are pairwise essentially disjoint.}
\end{enumerate} 
Indeed, each ideal $B_{i}$ is determined by the isometric embedding 
\[\phi|_{L^{1}(m^{\lambda_{i}})}\colon L^{1}(m^{\lambda_{i}})\hookrightarrow L^{1}(\mu)\quad \mathrm{for}\ i\in A_{j},\] 
namely, $B_{i}$ is the essential support of the functions in the image of $\phi|_{L^{1}(m^{\lambda_{i}})}$. 

Let $p\in [1,\infty)$. For each $i\in I$ there is a Boolean isomorphism from 
$\Sigma_{[0,1]^{\lambda_{i}}}\mod m^{\lambda_{i}}$ onto the corresponding ideal $B_{i}$ of $\Sigma^{L}$.
Clearly 
\[\{f\in L^{p}(\mu,\X)|\mathrm{supp}(f)\subset B_{i}\ \mu-\mathrm{a.e.}\}\subset L^{p}(\mu,\X)\]
is a closed subspace isometric to $L^{p}(\mu|_{B_{i}},\X)$. Thus there exists an isometric isomorphism 
$U_{i}\colon L^{p}(m^{\lambda_{i}},\X)\rightarrow \{f\in L^{p}(\mu,\X)|\mathrm{supp}(f)\subset B_{i}\ \mu-\mathrm{a.e.}\}$
for $i\in I$, see e.g. \cite[p.476]{GJK}. 

Define an isometric isomorphism $U\colon \bigoplus_{i\in I}^{p}L^{p}(m^{\lambda_{i}},\X)\rightarrow L^{p}(\mu,\X)$ by 
$(f_{i})_{i\in I}\mapsto \sum_{i\in I}U_{i}(f_{i})$. Indeed, this mapping is an isometry since the ideals $B_{i}$ are disjoint.
Moreover, as $\Sigma^{L}$ is $\sigma$-generated by the ideals $B_{i}$ it is easy to see by analyzing simple functions 
of $L^{p}(\mu,\X)$ that $U$ is onto. 

The following final step finishes the proof. We claim that if $\kappa$ is a non-zero cardinal, 
then $L^{p}(m^{\kappa},\X)$ is isometrically isomorphic to $L^{p}(m,L^{p}(m^{\kappa},\X))$. 
Indeed, recall that each $m\otimes m^{\kappa}$-measurable set 
$A\subset [0,1]\times [0,1]^{\kappa}$ can be approximated in measure 
by countable unions of measurable rectangles (see e.g. \cite[p.145]{Ha}). For the scalar-valued case, 
see \cite[p.127]{Lac} and \cite[p.478-479]{GJK}. Hence the obvious identification of the spaces 
$L^{p}(m\otimes m^{\kappa},\X)$ and $L^{p}(m,L^{p}(m^{\kappa},\X))$ can be obtained, since
each simple function in $L^{p}(m\otimes m^{\kappa},\X)$ can be approximated by a sequence of simple functions
of $L^{p}(m,L^{p}(m^{\kappa},\X))$ and vice versa.
\end{proof}

Now we are ready to prove our main result.

\begin{proof}[Proof of Theorem. \ref{thm1}]
Since $\mu$ is atomless we may apply Lemma \ref{thm2} to $L^{p}(\mu,\X)$ for $1\leq p<\infty$. Thus we may write 
$L^{p}(\mu,\X)=\bigoplus_{i\in I}^{p}L^{p}(m,L^{p}(m^{\kappa_{i}},\X))$ isometrically, 
where $\kappa_{i}$ are non-zero cardinals for $i\in I$. In what follows we will denote 
$\mathcal{M}\stackrel{\cdot}{=}\bigoplus_{i\in I}^{p}L^{p}(m,L^{p}(m^{\kappa_{i}},\X))$ for the sake of brevity.
Hence it suffices to show that $\mathcal{G}_{\mathcal{M}}$ is $\mathrm{SOT}$-contractible.

For each $i\in I$ we define mappings 
\[\alpha_{i},\beta_{i},\gamma_{i}\colon [0,1]\times L^{p}(m,L^{p}(m^{\kappa_{i}},\X))\rightarrow L^{p}(m,L^{p}(m^{\kappa_{i}},\X))\] 
by $\alpha_{i}(t,f_{i})=\chi_{[0,t]}f_{i}$, 
\[\beta_{i}(t,f_{i})(s)=(1-t)^{-\frac{1}{p}}f_{i}(t+(1-t)s)\ \mathrm{and}\ \gamma_{i}(t,f_{i})\circ \beta_{i}(t,f_{i})=\chi_{[t,1]}f_{i},\]
where $f_{i}\in L^{p}(m,L^{p}(m^{\kappa_{i}},\X))$ and $s\in [0,1]$.
Above we apply the convention $0^{-\frac{1}{p}}=0$.
Clearly $\alpha_{i}(t,\cdot),\beta_{i}(t,\cdot)$ and $\gamma_{i}(t,\cdot)$ are contractive linear operators for $i\in I,\ t\in [0,1]$. 
Note that $\alpha_{i}(\cdot,f_{i})$ are continuous on $[0,1]$ and according to Lemma \ref{fact2} the same is true for 
$\beta_{i}(\cdot,f_{i})$ and $\gamma_{i}(\cdot,f_{i})$. 
 
The required homotopy $h\colon [0,1]\times \mathcal{G}_{\mathcal{M}}\rightarrow \mathcal{G}_{\mathcal{M}}$ is given by 
\begin{equation*}
\begin{array}{cc}
&h(t,T)\left(\sum_{i\in I}f_{i}\right)=\sum_{i \in I}\alpha_{i}(t,f_{i})+
\sum_{i\in I}\gamma_{i}\left(t,P_{i}T\left(\sum_{i\in I}\beta_{i}(t,f_{i})\right)\right)
\end{array}
\end{equation*}
for $t\in [0,1],\ T\in \mathcal{G}_{\mathcal{M}}$
and $\sum_{i\in I}f_{i}\in \mathcal{M}$. 
Indeed, it is straightforward to check that $h(t,T)\in \mathcal{G}_{\mathcal{M}}$ 
for each $t\in [0,1]$ and $T\in \mathcal{G}_{\mathcal{M}}$.
Moreover, $h(0,T)=T$ and $h(1,T)=\mathrm{id}$ for each $T\in \mathcal{G}_{\mathcal{M}}$.
Our next aim is to justify that $h$ is indeed a suitable homotopy with respect to the strong operator topology.
 
Let $t_{0}\in [0,1]$, $T\in \mathcal{G}_{\mathcal{M}}$ and
let $V\subset \mathcal{G}_{\mathcal{M}}$ be a $\mathrm{SOT}$-open neighbourhood
of $h(t_{0},T)$. In order to justify the $|\cdot|\times\mathrm{SOT}-\mathrm{SOT}$ continuity of $h$ at $(t_{0},T)$, we must
find an open set $W\subset ([0,1],|\cdot|)\times (\mathcal{G}_{\mathcal{M}},\mathrm{SOT})$
such that $(t_{0},T)\in W$ and $h(W)\subset V$. Fix $\epsilon>0$ and 
$f=\sum_{i\in I}f_{i},g=\sum_{i\in I}g_{i}\in \S_{\mathcal{M}}$ 
such that $||g-h(t_{0},T)(f)||<\epsilon$. 

Recall that the point $h(t_{0},T)$ has a $\mathrm{SOT-}$neighbourhood basis of type 
\eqref{eq: SOTbasis}, and the elementary topological fact that open sets are preserved in finite intersections.  
Since $\epsilon$, $f$ and $g$ were arbitrary, it suffices to show that there are open sets $\Delta\subset ([0,1],|\cdot|)$ and 
$U\subset (\mathcal{G}_{\mathcal{M}},\mathrm{SOT})$
such that $t_{0}\in \Delta,\ T\in U$ and
\begin{equation}\label{eq: hdeltau}
h(\Delta\times U)\subset \{R\in \mathcal{G}_{\mathcal{M}}:\ ||g-Rf||<\epsilon\}.
\end{equation}

Denote $\delta=\epsilon-||g-h(t_{0},T)f||>0$. There exists a finite set $J\subset I$ such that 
$||\sum_{i\in I\setminus J}f_{i}||<\frac{\delta}{4}$. Put $n=|J|$. Fix $j\in J$. Similarly as above, let $f_{j}=P_{j}f$.
Since $\beta_{j}(\cdot,f_{j})$ is continuous we can find an open interval
$\Delta_{j}\subset [0,1]$ containing $t_{0}$ such that 
$||\beta_{j}(t,f_{j})-\beta_{j}(t_{0},f_{j})||< \frac{\delta}{4n}$ for $t\in \Delta_{j}$. 
Define an $\mathrm{SOT}-$open neighbourhood $U_{j}$ of $T$ by
\[U_{j}=\{S\in \mathcal{G}_{\mathcal{M}}:\ ||(S-T)\beta_{j}(t_{0},f_{j})||<\frac{\delta}{4n}\}.\] 

Finally, by recalling the definitions of $h$, $\Delta_{j}$ and $U_{j}$ we obtain that 
\begin{equation*}
\begin{array}{ll}
||h(t,S)f_{j}-h(t_{0},T)f_{j}||&\leq ||h(t,S)f_{j}-h(t_{0},S)f_{j}||+||h(t_{0},S)f_{j}-h(t_{0},T)f_{j}||\\
&<\frac{\delta}{4n}+||(S-T)\beta_{j}(t_{0},f_{j})||<\frac{\delta}{2n}
\end{array}
\end{equation*}
for $t\in \Delta_{j}$ and $S\in U_{j}$, where $j\in J$. Put $\Delta=\bigcap_{j\in J}\Delta_{j}$ and $U=\bigcap_{j\in J}U_{j}$.
We get
\begin{flushleft}
$||h(t,S)f-h(t_{0},T)f||$\\
$\leq \left|\left|(h(t,S)-h(t_{0},T))\sum_{i\in I\setminus J}f_{i}\right|\right|+\sum_{j\in J}||h(t,S)f_{j}-h(t_{0},T)f_{j}||$\\
$<||h(t,S)-h(t_{0},T)||\ \cdot\ \left|\left|\sum_{I\setminus J}f_{i}\right|\right|+n\frac{\delta}{2n}< \delta$
\end{flushleft}
for $(t,S)\in \Delta\times U$. This means that $||g-h(t,S)f||<||g-h(t_{0},T)f||+\delta=\epsilon$ for $(t,S)\in \Delta\times U$.
Consequently $h(\Delta \times U)$ satisfies \eqref{eq: hdeltau} and the proof is complete.
\end{proof}

Recall that $\S_{L^{p}(m)},\ 1\leq p<\infty,$ consists of (exactly) $2$ orbits, namely
$\mathcal{G}_{L^{p}}(\chi_{[0,1]})$ and $\mathcal{G}_{L^{p}}(2^{\frac{1}{p}}\chi_{[0,\frac{1}{2}]})$,
both of which are dense in $\S_{L^{p}(m)}$ (see e.g. \cite{Rol}).  
\begin{corollary}
Both the orbits 
$\mathcal{G}_{L^{p}}(\chi_{[0,1]}),\mathcal{G}_{L^{p}}(2^{\frac{1}{p}}\chi_{[0,\frac{1}{2}]}),\ 1\leq p<\infty,$ 
are path-connected. 
\end{corollary} 
\begin{proof}
Fix $2$ points $f,g\in \S_{\X}$ both coming from one of the above orbits. Then there exists $T\in \mathcal{G}_{L^{p}}$
such that $T(f)=g$. According to Theorem \ref{thm1} there is a homotopy $h\colon [0,1]\times \mathcal{G}_{L^{p}}\rightarrow \mathcal{G}_{L^{p}}$ such that $h(0,T)=T$ and $h(1,T)=\mathrm{id}$. Clearly $h([0,1]\times \mathcal{G}_{L^{p}})(\cdot)$ preserves orbits.
Note that $t\mapsto h(t,T)f$ defines a continuous path connecting $g$ and $f$ in $\mathcal{G}_{L^{p}}(f)\subset \S_{L^{p}}$. 
\end{proof}

The assumptions of $\mu$ being atomless or $p<\infty$ cannot be removed in Theorem \ref{thm1} even in the scalar-valued case.
If $p\in [1,\infty),\ p\neq 2,$ and $\mu$ has some atoms, then by applying the scalar-valued analogue of representatation 
\eqref{eq: *} (see \cite{Lac}) it can be verified that $\mathcal{G}_{L^{p}(\mu)}$ is not connected in the weak operator topology. 
However, if $\nu$ is the counting measure on $\N$, then $L^{p}(\nu,L^{p}(m))=L^{p}(m)$ isometrically,
whose isometry group is $\mathrm{SOT-}$contractible according to Theorem \ref{thm1}. 

\begin{proposition}
Let $(\Omega,\Sigma,\mu)$ be a positive measure space and we will regard $L^{\infty}(\mu)$ over the real field.
Then $\mathcal{G}_{L^{\infty}(\mu)}$ is totally separated in the strong operator topology.
\end{proposition}
\begin{proof} 
For given $T,S\in \mathcal{G}_{L^{\infty}(\mu)}$ and $A\in \Sigma$ it holds that $T(\chi_{A}) \neq S(\chi_{A})$ if and only if 
$||T\chi_{A}-S\chi_{A}||\geq 1$ (see e.g. \cite[Thm. 2]{Gr_infty}). 
Pick $T,S\in \mathcal{G}_{L^{\infty}(\mu)},\ T\neq S$, if such exist. It is easy to find a set $C\in \Sigma$ such that 
$T\chi_{C}\neq S\chi_{C}$. Now, 
\[\{U\in \mathcal{G}_{L^{\infty}(\mu)}:\ ||U\chi_{C}-T\chi_{C}||<\frac{1}{2}\}\cup \{V\in \mathcal{G}_{L^{\infty}(\mu)}:\ ||U\chi_{C}-T\chi_{C}||>\frac{1}{2}\}=\mathcal{G}_{L^{\infty}(\mu)}\]
is $\mathrm{SOT-}$separation completing the claim. 
\end{proof}

To conclude, let us make a few remarks about the homotopy $h$ appearing in the proof of Theorem \ref{thm1}.
We note that a resembling transformation was applied in \cite[p.251]{DD} in a different setting. 
The proof of Theorem \ref{thm1} can easily be modified so that $h$ becomes a homotopy on the set of linear isometric 
embeddings $L^{p}(\mu,\X)\rightarrow L^{p}(\mu,\X)$, and hence this set is $\mathrm{SOT}$-contractible 
for atomless $\mu$ and $p<\infty$. With equally small modifications it can be verified that the conclusion of Theorem \ref{thm1} 
remains valid, if one investigates the contractibility of $\mathrm{GL}_{L^{p}(\mu,\X)}$ in place of 
$\mathcal{G}_{L^{p}(\mu,\X)}$.

\end{document}